\documentclass[%
 aip,
 jmp,%
 amsmath,amssymb,
 reprint,%
onecolumn]{revtex4-1}

\usepackage{graphicx}
\usepackage{bm}

\usepackage{latexsym,amssymb,amsmath,amsopn,graphics,xy,epsfig,picture,epic}

\def\aa{{\mathcal A}}
\def\bb{{\mathcal B}}
\def\cc{{\mathcal C}}

\def\ff{{\mathcal F}}
\def\hh{{\mathcal H}}

\def\ll{{\mathcal L}}

\def\qq{{\mathcal Q}}
\def\rr{{\mathcal R}}
\def\ss{{\mathcal S}}
\def\tt{{\mathcal T}}

\def\ffi{\varphi}
\def\eps{\varepsilon}
\def\dst{\displaystyle}

\renewcommand{\Im}{\DeclareMathOperator{\Im}{Im}}

\DeclareMathOperator{\tr}{tr}
\DeclareMathOperator{\supp}{supp}
\DeclareMathOperator{\sgn}{sgn}
\DeclareMathOperator{\cotan}{cotan}

\def\C{{\mathbb{C}}}

\def\N{{\mathbb{N}}}

\def\Q{{\mathbb{Q}}}
\def\R{{\mathbb{R}}}

\def\T{{\mathbb{T}}}

\def\Z{{\mathbb{Z}}}

\newcommand{\norm}[1]{{\left\|{#1}\right\|}}
\newcommand{\ent}[1]{{\left[{#1}\right]}}
\newcommand{\abs}[1]{{\left|{#1}\right|}}
\newcommand{\scal}[1]{{\left\langle{#1}\right\rangle}}
\newcommand{\set}[1]{{\left\{{#1}\right\}}}

\newenvironment{definition}[1][]{\vskip3pt\noindent\sl\textbf{Definition.}\ }{\rm\vskip3pt}
\newenvironment{remark}[1][]{\vskip3pt\noindent\textbf{Remark.}\ }{\rm\vskip3pt}

\newenvironment{proof}[1][]{\vskip3pt\noindent\textit{Proof.}\ }{\rm\vskip3pt}

\newtheorem{lemma}{Lemma}[section]
\newtheorem{proposition}[lemma]{Proposition}
\newtheorem{theorem}[lemma]{Theorem}
\newtheorem{corollary}[lemma]{Corollary}

\date{\today}

\newcommand{\ieg}{\leqslant}
\newcommand{\de}{\mathrm{d}}

\begin{document}
\title[Generalized Pauli Problem]{Zak Transform and non-uniqueness in an extension of Pauli's phase retrieval problem}

\author{Simon Andreys \& Philippe Jaming}

\begin{abstract}
The aim of this paper is to pursue the investigation of the phase retrieval problem for the fractional Fourier transform $\ff_\alpha$
started by the second author.
We here extend a method of A.E.J.M Janssen to
show that there is a countable set $\qq$ such that for every finite subset $\aa\subset \qq$, there exist two
functions $f,g$ not multiple of one an other such that $|\ff_\alpha f|=|\ff_\alpha g|$ for every $\alpha\in \aa$.
Equivalently, in quantum mechanics, this result reformulates as follows:
if $Q_\alpha=Q\cos\alpha+P\sin\alpha$ ($Q,P$ be the position and momentum observables),
then $\{Q_\alpha,\alpha\in\aa\}$ is not informationally complete with respect to pure states.

This is done by constructing two functions $\ffi,\psi$ such that $\ff_\alpha\ffi$ and $\ff_\alpha\psi$ have disjoint support for each 
$\alpha\in \aa$. To do so, we establish a link between $\ff_\alpha[f]$, $\alpha\in \qq$ and the Zak transform $Z[f]$
generalizing the well known marginal properties of $Z$.
\end{abstract}

\keywords{Zak transform; Weyl-Heisenberg transform; Fractional Fourier Transform; Phase Retrieval; Pauli problem}

\address{Institut de Math\'ematiques de Bordeaux UMR 5251,
Universit\'e de Bordeaux, cours de la Lib\'eration, F 33405 Talence cedex, France}
\email{Philippe.Jaming@gmail.com}

\maketitle

\section{Introduction}

Phase Retrieval Problems arise in many aspects of applied science including optics and quantum mechanics.
The problem consists of reconstructing a function $\ffi$ from phase-less measurements
of some transformations of $\ffi$ and eventual a priori knowledge on $\ffi$. 
Recently, this family of problems has attracted a lot of attention in the mathematical
community ({\it see e.g.} \cite{candes,candes2,mallat} or the research blog \cite{Mi}).
A typical such problem
is the phase retrieval problem in optics where one wants to reconstruct a compactly supported function
$\ffi$ from the modulus of its Fourier transform $|\ff[\ffi]|$. An other such problem, due to Pauli,
is to reconstruct $\ffi$ from its modulus $|\ffi|$ and that of its Fourier transform $|\ff[\ffi]|$.
The questions has two sides. On one hand, one wants to know if such a problem has a unique solution
(up to a global phase factor or a more general transformation). On the other hand, one is looking
for an algorithm that allows to reconstruct $f$ from phase-less measurements.

In this paper, we focus on the (non) uniqueness aspects of a generalization of Pauli's problem
in which one has phase-less measurements of Fractional Fourier Transform (FrFT)
as has been previously studied by the second author \cite{jaming}.
Recall from \cite{Na,OZK} that the FrFT is defined as follows:
for $\alpha\in\R\setminus\pi\Z$ a parameter that is interpreted as an angle, and $f\in L^1(\R)\cap L^2(\R)$,
$$
\ff_\alpha f(\xi)
=c_\alpha e^{-i\pi|\xi|^2\cot\alpha}\ff[e^{-i\pi|\cdot|^2\cot\alpha}f](\xi/\sin\alpha)
$$
where $c_\alpha$ is a normalization constant which ensures that $\ff_\alpha$ extends to a unitary isomorphism of $L^2(\R)$.
Note that $\ff_{\pi/2}=\ff$.
One then defines $\ff_{k\pi}[\ffi](\xi)=\ffi\bigl((-1)^k\xi\bigr)$ so that $\ff_\alpha\ff_\beta=\ff_{\alpha+\beta}$.
Roughly speaking,
the Fractional Fourier Transforms are adapted to the mathematical expression of the Fresnel
diffraction, just as the standard Fourier transform is adapted to Fraunhofer diffraction.
It also occurs in quantum mechanics (see below).

The question we want to address here is the following:

\medskip

\noindent{\bf Phase Retrieval Problem for Multiple Fractional Fourier Transforms.}\\
{\sl Let $\aa\subset(0,\pi)$. For $f,g\in L^2(\R)$ does
$|\ff_\alpha[g]|=|\ff_\alpha[f]|$ for every $\alpha\in\aa$ imply that there is a constant $c\in\C$ with $|c|=1$
such that $g=cf$.

If this is the case, we will say that $\{Q_\alpha\,:\alpha\in\aa\}$ is \emph{Informationnaly Complete with respect to Pure States}
(ICPS).}

\medskip

The vocabulary comes from quantum mechanics ({\it see e.g.} \cite{CHST}).
For sake of self-containment, let us now  recall some basic notions from quantum mechanics. 
The description of a physical system is based on a complex separable Hilbert space $\mathcal{H}$.
For our purpuses, $\mathcal{H}=L^2(\R)$. We will denote by $\bb=\bb(\hh)$ and $\tt=\tt(\hh)$ the bounded and trace class operators on $\hh$.
The state of the system is then represented by an element 
$\rho$ of $\tt(\hh)$ that is positive $\rho\geq0$ and normalized by $\tr[\rho]=1$.
The states form a convex set whose extreme points, the \emph{pure states} are the one dimensional projections.

The \emph{observables} are normalized Positive Operator Valued Measure $E\,:\bb(\R)\to\ll(\hh)$ where $\bb(\R)$
is the Borel set $\R$, that is, maps satisfying $E[X]\geq 0$, $E[\R]=I$,
and, for every $\rho\in\tt(\hh)$, for every sequence $(X_j)$ of pairwise disjoint sets in $\bb(\R)$,
$\tr[\rho E(\bigcup_jX_j)]=\sum_j\tr[\rho E(X_j)]$. It follows that, for any state $\rho$, $\rho^E=\tr[\rho E]$
is a probability measure on $\R$. The number $\rho^E(X)$ is interpreted as the probability that the measurement of $E$
gives an outcome from the set $X$ when the system is initially prepared in the state $\rho$.

The problem we are concerned here is thus following:

\begin{definition}
Let $\aa$ be a collection of observables $\bb(\R)\to\ll(\hh)$. Then

-- $\aa$ is informationally complete if, for any
$\rho_1,\rho_2\in\tt(\hh)$, $\rho_1^E=\rho_2^E$ for all $E\in\aa$
implies $\rho_1=\rho_2$;

-- $\aa$ is informationally complete \emph{with respect to pure states} if, for any pure states
$\rho_1,\rho_2$, $\rho_1^E=\rho_2^E$ for all $E\in\aa$
implies $\rho_1=\rho_2$.
\end{definition}

Now let $Q,P$ be the position and momentum observables. For any $\alpha\in[0,2\pi]$, define
$Q_\alpha(X)=\ff_{-\alpha}Q\ff_\alpha=Q\cos\alpha+P\sin\alpha$.
For a system in a pure state $\rho=|\psi\rangle\langle\psi|$, the measurement outcome probabilities are given by
$$
\rho^{Q_\alpha}(X)=\scal{\psi,Q_\alpha(X)\psi}=\int_X|\ff_\alpha\psi(x)|^2\,\de x.
$$
This shows that the probability density associated to a measurement of the observable $Q_\alpha$ preformed on a pure
state $|\psi\rangle\langle\psi|$ is just the intensity $|\ff_\alpha\psi(x)|^2$.

\medskip

In this setting, Pauli conjectered that $\{0,\pi/2\}$ is informationnaly complete with respect to pure states.
However, this was soon disproved and there are now a fairly large number of counterexamples
({\it see e.g.} \cite{jamingjfaa}, \cite{Co} and references therein). In \cite{Vo}, A. Vogt mentioned 
a conjecture by Wright that there exists a
third observables such that this observable together with position and momentum forms an informationnaly complete with respect to pure states family. This amounts to looking for a unitary operator $U_3\,:L^2(\R)\to L^2(\R)$ such that, if we set
$U_1=I$, $U_2=\ff$ then $|U_jf|=|U_jg|$ if and only if $f=cg$ with $|c|=1$.
In \cite{jaming}, the second author investigated the above question further. In particular, when
$\alpha\notin\Q\pi$ none of the classes of counterexamples 
mentioned in \cite{jamingjfaa,Co} allowed to show that $\aa:=\{Q(\theta), \theta=0,\alpha,\pi/2\}$
would not be ICPS. This lead naturally to conjecture that $\aa$ was ICPS, which
was recently disproved by Carmeli, Heinosaari, Schultz \& A. Toigo in \cite{CHST}. Actually, those authors
exhibit an invariance property which shows that the ICPS property of $\{Q(\alpha_1),Q(\alpha_2),Q(\alpha_3)\}$
does not depend on $\alpha_1,\alpha_2,\alpha_3$ provided they are different. Then they extend a construction
of \cite{jaming} to show that $\{Q_\alpha,\alpha\in F\}$ is not ICPS when $F\subset\Q\pi$ is finite.

The aim of this paper is to provide a somewhat similar result for an other class of parameters
which is based on an adaptation of an unnoticed counterexample to Pauli's problem due to A.E.J.M. Janssen \cite{janss2}:

\medskip

\noindent{\bf Main Theorem.}\ \\
{\sl 
Let $F\subset\Q$ be a finite set of rational numbers. Define $\aa$ by
$$
\aa=\{\alpha\in (0,\pi)\,:\mbox{ there exists }r\in F\mbox{ with }\cot\alpha=r\}.
$$
Then the collection $\{Q_{\alpha},\alpha\in\aa \}$ is not Informationnaly Complete with
respect to Pure States.}

\medskip

This is based on the fact that a function and its Fourier transform $f$ are recovered from its Zak transform (also known as Weyl-Heisenberg tranform)
\cite{zak67,Gro}
$$
Z[f](x,\xi)=\sum_{k\in\Z} f(x+k)e^{-2i\pi k\xi}.
$$
by integrating along horizontal and vertical lines. We adapt this property by showing that, when $\cot\alpha\in\Q$,
then $\ff_\alpha f$ can also be recovered from $Z[f]$ by integrating in an oblique direction.
We believe this result is interesting in itself.

Once this is done, one can construct two functions $f_1$ and $f_2$ such that for each $\alpha\in\aa$ ($\aa$ defined in the Main Theorem),
$\ff_\alpha[f_1]$ and $\ff_\alpha[f_2]$ have disjoint support. Therefore, for any $c\in\C$ with $|c|=1$, and any $\alpha\in\aa$,
$|\ff_\alpha[f_1+cf_2]|=|\ff_\alpha[f_1]|+|\ff_\alpha[f_2]|$ does not depend on $c$,
but $f_1+cf_2$ is not a constant
multiple of $f_1+f_2$ unless $c=1$ thus proving the main theorem.

\medskip

We complete the paper by showing that the approximate phase retrieval problem for the FrFT has infinitely many solutions that
are far from multiples from one an other:

\medskip

\noindent{\bf The Approximate Phase Retrieval Problem for the Fractional Fourier Transform.}\\
{\sl Leq $0\leq \alpha_1< \cdots<\alpha_n \leq \frac{\pi}{2} $ and $T>0$.
Let $f_1, ... f_n \in L^\infty(\R)$ with $\supp f_j\subset\ent{-T,T}$ and $f_j\geq0$. Then, for every $\eps>0$, there exists 
$\ffi_1,\ldots,\ffi_n \in L^2(\R)$ such that, for $c_1,\ldots,c_n\in\C$ with $|c_1|=\cdots=|c_n|=1$,
$\ffi=\dst\sum_{j=1}^nc_j\ffi_j$ satisfies}
\begin{equation}
\label{eq:approxintro}
\norm{\abs{\ff_{\alpha_k}[\ffi]}-f_k}_{L^\infty(\ent{-T,T})} \leq \eps\qquad\mbox{for }k=1,\ldots, n.
\end{equation}

%
%

\medskip

The remaining of the paper is organized as follows. In the next section, we first make precise definitions of the FrFT and of the Zak 
Transform and then show the oblique marginal property of the Zak transform in Theorem \ref{th:zakmargin}. Section \ref{sec:pauli} is then devoted to
the proof of the main theorem.
The last section is devoted to proving our result on the approximate phase retrieval problem.

\section{The main theorem}

\subsection{The Zak Transform and the Fractional Fourier Transform}

Let $\alpha\notin \pi \Z$, the Fractional Fourier Transform of order $\alpha$ of $f\in L^1(\R)$ is defined as
\begin{eqnarray*}
\ff_\alpha f(\xi)
&=&c_\alpha e^{-i\pi|\xi|^2\cot\alpha}\ff[e^{-i\pi|\cdot|^2\cot\alpha}f](\xi/\sin\alpha) \\
&=& c_\alpha e^{-i\pi \cot{\alpha}\ \xi^2} \int_\R e^{-i\pi/2 \cot{\alpha} \ x^2 -2i \pi \frac{x \xi}{\sin{\alpha} }  } f(x) \ \de x
\end{eqnarray*}
where
$c_\alpha=\frac{\exp{\frac{i}{2}\left(\alpha-\text{sgn}(\alpha) \frac{\pi}{4}  \right)}}{\sqrt{\abs{\sin{\alpha}}}}$
is such that $c_\alpha ^2 =1-i\cot{\alpha}$ and $\rr e\ c_\alpha > 0$.
For $\alpha =k \pi$, $k\in\Z$, we set $\ff_\alpha g (\xi)=g\bigl((-1)^k \xi\bigr)$.

Notice that $\ff_{\frac{\pi}{2}}= \ff$. The FrFT further has the following properties

\begin{enumerate}
\item For $u\in\C$ let $\gamma_u(t)=e^{-u\pi t^2}$ then 
$\ff_\alpha[f](\xi)=c_\alpha \gamma_{i\cot\alpha}(\xi)\ff[\gamma_{i\cot\alpha}f](-\xi/\sin\alpha)$;
\item for every $\alpha\in\R$, and $f\in L^1(\R)\cap L^2(\R)$, $\norm{\ff_\alpha f}_{L^2(\R)}=\norm{f}_{L^2(\R)}$
thus $f$ extends into a unitary operator on $L^2(\R)$;
\item $\ff_\alpha \circ \ff_\beta=\ff_{\alpha+\beta}$.
\end{enumerate}

The Zak transform is defined, for $f\in\ss(\R)$ (the Schwarz class) as
$$
Zf(x,\xi)=\sum_{k\in\Z} f(x+k)e^{-2i\pi k\xi}.
$$
The Zak transform has the following properties
\begin{enumerate}
\item $Z$ extends into a unitary operator $L^2(\R)\to L^2(Q)$ where $Q=[0,1]\times [0,1]$.

\item $Zf(x+n,\xi)=  e^{2 i \pi n \xi} Zf(x,\xi)$ and $Zf(x,\xi+n)  =  Zf(x,\xi)$.
 
\item Poisson summation: $Zf(x,\xi)=e^{2i\pi x \xi } Z \hat{f} (\xi,-x)$ 

\item $Z$ has the following marginal properties:
$$
f(x)= \int_0^1 Zf(x,\xi)\ \de\xi\quad \mbox{and}\quad
\ff f(\xi)=\int_0^1 e^{-2i\pi x \xi} Zf(x,\xi)\ \de x.
$$
\end{enumerate}

We can now prove that the Zak transform has also marginal properties in oblique directions linked to the
Fractional Fourier Transform. Not surprisingly, such a marginal property only occurs when the slope is rational,
as otherwise one would integrate over a dense subset of $Q$.

In order to establish such a property we will first need to compute
the Zak transform of a chirp. This has to be done in the sense of distributions:


\begin{proposition}[Zak transform of a chirp]\label{prop:zakchirp} \ \\
Let $p\in\Z$, $q\in\N$ be relatively prime integers.
For $n\in\Z$, let
\[ 
c_{n,p,q}=\frac{1}{q} \sum_{k=0}^{q-1} (-1)^{kp} e^{i \pi \frac{p}{q} k^2}e^{-2i \pi \frac{nk}{q} }
\]
and, for $x\in [0,1]$, $n\in\Z$, write $\xi_{n,p,q}(x)=\frac{p}{q}x +\frac{1}{2} p +\frac{n}{q}$ and  
$A_{p,q}(x)=\set{n \in \Z |\ 0 \leq \xi_{n,p,q}(x) \leq 1}$.

For $f\in \ss(\R)$,
\begin{equation}
\label{eq:zak}
\int_{\R} f(t) e^{-i\frac{p}{q}t^2}\,\de t =  \int_0^1 e^{-i \pi \frac{p}{q} x^2} \sum_{ n \in A_{p,q}(x) } 
c_{n,p,q}\  Zf \bigl(x,\xi_{n,p,q}(x) \bigr)\,\de x.
\end{equation}
\end{proposition}

\begin{remark}
This formula is stated in \cite{jans} without proof in the following form
\begin{equation}
\label{eq:zakdistri}
Z\gamma_{-ip/q} (x,\xi)=\gamma_{-ip/q} \sum_{n \in \Z} c_{n,p,q} \delta\left(\xi-\frac{p}{q}x -\frac{1}{2} p -\frac{n}{q} \right)
\end{equation}
where $\delta$ is the Dirac delta function.
When interpreted in the sense of tempered distributions, this formula
is precisely \eqref{eq:zak}. Our first aim here is to provide a rigorous proof of this formula.
\end{remark}

\begin{proof}
Define 
\[
Z_{N,N'} \gamma_{-ip/q} (x,\xi)= \sum_{k=-N}^{N'} \gamma_{-ip/q}(x+k) e^{-2i\pi k \xi}.
\]
Let $f\in\ss(\R)$. Note that the series defining the Zak transform of $f$ is uniformly convergent, thus
\begin{eqnarray*}
\scal{Zf,Z_{N,N'} \gamma_{-ip/q}}_{L^2(Q)}&=&
\int_0^1\int_0^1\sum_{j\in\Z}\sum_{k=-N}^{N'} f(x+j)\gamma_{ip/q}(x+k) e^{-2i\pi j \xi}e^{2i\pi k \xi}\,\mbox{d}\xi\,\mbox{d}x\\
&=&\sum_{k=-N}^{N'}\int_0^1\sum_{j\in\Z} f(x+j)\gamma_{ip/q}(x+k) \int_0^1e^{2i\pi (k-j) \xi}\,\mbox{d}\xi\,\mbox{d}x\\
&=&\sum_{k=-N}^{N'}\int_0^1 \gamma_{ip/q}(x+k)f(x+k)\,\mbox{d}x\\
&=&\int_{-N}^{N'+1}  f(t)\gamma_{ip/q}(t)\,\mbox{d}t. 
\end{eqnarray*}

It follows that
\begin{equation}
\label{eq:lim1}
\lim_{N,N'\to+\infty}\scal{Zf,Z_{N,N'} \gamma_{-ip/q}}_{L^2(Q)}
=\int_{\R}  f(t)\gamma_{ip/q}(t)\,\mbox{d}t.
\end{equation}

On the other hand,
\begin{eqnarray*}
Z_{N,N'} \gamma_{-ip/q} (x,\xi)&=& \sum_{j=-N}^{N'} e^{i\pi \frac{p}{q} (x+j)^2-2i\pi j \xi}\\
             &=& e^{i\pi \frac{p}{q} x^2} \sum_{j=-N}^{N'} e^{i \pi \frac{p}{q} j^2} e^{2i\pi j\left(\frac{p}{q}  x - \xi\right)}.
\end{eqnarray*}
We then split this sum according to the value of $j$ modulo $2q$.
In order to do so, note that, if $j=2ql+k$, $\pi\frac{p}{q} j^2=\frac{p}{q} k^2 \pmod {2\pi}$
so that $e^{i \pi \frac{p}{q} j^2}=e^{i \pi \frac{p}{q} k^2}$. 

Without loss of generality, we will now assume that$N=2qM$ and $N'=2q(M+1)-1$. 
Then, every $j\in\{-N,\ldots,N'\}$ decomposes uniquely as $j=2q\ell+k$ with $\ell\in\{-M,\ldots,M\}$
and $k\in\{0,\ldots,2q-1\}$. Therefore
\begin{eqnarray*}
Z_{N,N'} \gamma_{-ip/q}(x,\xi)
	&=&e^{i\pi \frac{p}{q} x^2} \sum_{\ell=-M}^M \sum_{k=0}^{2q-1} e^{i \pi \frac{p}{q} k^2} 
	e^{2i\pi  (2q\ell+k)\left(\frac{p}{q}x - \xi\right)} \\
  &=&e^{i\pi \frac{p}{q} x^2}\sum_{k=0}^{2q-1} e^{i \pi \left(\frac{p}{q} k^2+ 2 k \left(\frac{p}{q}x-\xi\right) \right)}
		\sum_{\ell=-M}^M e^{4 i\pi \ell(px - q\xi)}.
\end{eqnarray*}
We are thus lead to introduce the Dirichlet kernel
\[
D_M(u)=\sum_{\ell=-M}^M e^{2i\pi u}
\]
and
\[
P(x,\xi)=e^{-i\pi \frac{p}{q} x^2}\sum_{k=0}^{2q-1} e^{-i \pi \left(\frac{p}{q} k^2+ 2 k \left(\frac{p}{q}x-\xi\right) \right)}.
\]
so that
$$
\overline{Z_{N,N'} \gamma_{-ip/q}(x,\xi)}=P(x,\xi),D_M\bigl(2(px-q\xi)\bigr)
$$

We have thus established that, if $N=2qM$, $N'=2q(M+1)-1$ and $f\in \ss(\R)$, then
\begin{eqnarray}
\scal{Zf,Z_{N,N'} \gamma_{ip/q}}_{L^2(Q)}
&=&\iint_{[0,1]^2}  Zf(x,\xi)\overline{Z_N \gamma_{ip/q}(x,\xi)}\,\de x\,\de\xi\nonumber\\
&=&\iint_{[0,1]^2} Zf(x,\xi)  P(x,\xi)\,D_M\bigl(2(px-q\xi)\bigr)\,\de x\,\de\xi
\label{eq:lim2}
\end{eqnarray}

Write $\ffi(x,\xi)=Zf(x,\xi)  P(x,\xi)$. Note that, as $f\in\ss(\R)$, $\ffi\in\cc^\infty(\R^2)$.
We now want to evaluate the limit of
$$
\iint_{Q} \varphi(x,\xi) D_M(2(px-q\xi))\,\mbox{d}x\,\mbox{d}\xi
$$
when $M \rightarrow \infty$.
Let us first compute the integral with respect to $\xi$:
\begin{eqnarray*}
\int_0 ^1 \varphi(x,\xi) D_M(2(px-q\xi))\,\de \xi 
&=& \frac{1}{2q}\int_0^{2q} \varphi\left(x,\frac{\zeta}{2q} \right) D_M(2px - \zeta)\ ,\de\zeta \\
&=& \sum_{\ell=0}^{2q-1} \frac{1}{2q} \int_\ell^{\ell+1} \varphi \left( x, \frac{\zeta}{2q} \right) D_M(2px-\zeta)\,\de\zeta \\
&=& \sum_{\ell=0}^{2q-1} \frac{1}{2q} \int_0^1 \varphi\left( x,\frac{\zeta+\ell}{2q} \right) D_M(2px -\zeta)\,\de \zeta .
\end{eqnarray*}
As $D_M$ is $1$-periodic, we may remplace $2px$ by $2px-\lfloor 2px \rfloor$, thus
\begin{eqnarray*}
\int_0 ^1 \varphi(x,\xi) D_M(2(px-q\xi))\,\de\xi 
&=& \sum_{\ell=0}^{2q-1} \frac{1}{2q} \int_0^1 
\varphi\left( x,\frac{\zeta+\ell}{2q} \right) D_M(2px-\lfloor 2px \rfloor -\zeta)\,\de\zeta \\
&=& \left(\frac{1}{2q}\sum_{\ell=0}^{2q-1}   \varphi\left(x,\frac{\zeta+\ell}{2q}\right) \right)*_{\zeta} D_M  (2px-\lfloor 2px \rfloor)
\end{eqnarray*}
where $*$ is the (circular) convolution on $[0,1]$ and where we have used the parity of $D_M$.

But now, for every $x$, define $\psi_x=\dst\frac{1}{2q}\sum_{\ell=0}^{2q-1}   \varphi\left(x,\frac{\zeta+\ell}{2q}\right)$
and note that $\psi_x$ is a smooth $1$-periodic function. Moreover, $\psi_x * D_M$ is the partial sum of order $M$
of the Fourier series of $\psi_x$. 
Therefore $\psi_x*D_M\to\psi_x$ uniformly, in particular at the point $2px-\lfloor 2px \rfloor$.
It follows that
\begin{eqnarray*}
\int_0 ^1 \varphi(x,\xi) D_M(2(px-q\xi))d\xi   &\underset{M\to+\infty}{\longrightarrow}&
\sum_{l=0}^{2q-1} \frac{1}{2q} \ffi \left(x,\frac{2px-\lfloor 2px \rfloor +l}{2q} \right)\\
&=&\sum_{l=-\lfloor 2px \rfloor}^{2q-1-\lfloor 2px \rfloor} \frac{1}{2q} \ffi \left(x,\frac{2px +l}{2q} \right).
\end{eqnarray*}

Replacing $\ffi$ by $Zf\,P$ and using \eqref{eq:lim2} we have thus established that
\begin{equation}
\scal{Zf,Z_N \gamma_{-ip/q}}_{L^2(Q)}\underset{M\to+\infty}{\longrightarrow}
\int_0^1 \sum_{l=-\lfloor 2px \rfloor}^{2q-1-\lfloor 2px \rfloor} \frac{1}{2q}
Zf\left(x,\frac{2px +l}{2q}\right)  P\left(x,\frac{2px +l}{2q}\right) \ \de x.
\label{eq:lim3}
\end{equation}

We will now reshape this formula. Let us first simplify $P$:
\begin{eqnarray*}
e^{i\pi \frac{p}{q} x^2}P\left(x,\frac{2px-\ell}{2q}\right)&=& \sum_{k=0}^{2q-1} e^{i \pi \left(\frac{p}{q} k^2 - \frac{k \ell}{q} \right)}
=
\sum_{k=0}^{q-1}+\sum_{k=q}^{2q-1} e^{i \pi \left(\frac{p}{q} k^2-  \frac{k \ell}{q} \right)} \\
&=& \sum_{k=0}^{q-1} \left(  e^{i \pi \left(\frac{p}{q} k^2 -  \frac{k \ell}{q} \right)} 
+  e^{i \pi \left(\frac{p}{q} (k+q)^2 -  \frac{(k+q)k \ell}{q} \right)}  \right) \\
&=& \sum_{k=0}^{q-1} e^{i \pi \left(\frac{p}{q} k^2 -  \frac{k \ell}{q} \right)}
\left( 1+ e^{i \pi(pq - l)} \right).
\end{eqnarray*}
As
$$
1+ e^{i \pi(pq-l)}=1+(-1)^{pq+l}=\begin{cases}2&\mbox{if }l=pq \pmod{2}\\
0&\mbox{otherwise}\end{cases},
$$
and, writing $l-pq=2n$, we have $\frac{p}{q} k^2 -  \frac{k \ell}{q} =\frac{p}{q} k^2 -  2\frac{n k}{q} - kp$,
we deduce that
\[
P\left(x,\frac{2px-\ell}{2q}\right)= e^{-i\pi \frac{p}{q} x^2}\sum_{k=0}^{q-1} 2(-1)^{kp} e^{i\pi \left( \frac{p}{q} k^2 -  2\frac{n k}{q} \right)} =2q c_{n,p,q}e^{-i\pi \frac{p}{q} x^2}
\] 
Finally, the condition
$-\lfloor 2px\rfloor \ieg k \ieg 2q-1-\lfloor 2px\rfloor$ is equivalent to $\xi_{n,p,q}(x):=\frac{p}{q}x+\frac{l}{2}\in[0,1]$.
Therefore \eqref{eq:lim3} reads
\begin{equation}
\label{eq:lim4}
\scal{Zf,Z_N \gamma_{-ip/q}}_{L^2(Q)}\underset{M\to+\infty}{\longrightarrow}
  \int_0^1 e^{-i \pi \frac{p}{q} x^2} \sum_{ n \in A_{p,q}(x) } 
c_{n,p,q}\  Zf \bigl(x,\xi_{n,p,q}(x) \bigr)\,\de x.
\end{equation}
It remains to compare this equation with \eqref{eq:lim1} to obtain that,
for $f\in\ss(\R)$,
$$
\int_{\R}  f(t)e^{-i\frac{p}{q}t^2}\,\mbox{d}t=
\int_0^1 e^{-i \pi \frac{p}{q} x^2} \sum_{ n \in A_{p,q}(x) } 
c_{n,p,q}\  Zf \left(x,\xi_{n,p,q}(x) \right)\ \de x,
$$
as announced.
\end{proof}

We can now prove our main theorem

\begin{theorem}\label{th:zakmargin}\ \\
Let $p\in\Z\setminus\{0\}$, $q\in\N$ be relatively prime integers and let $\alpha$ be defined by $\cot\alpha=\frac{p}{q}$.
For $n\in\Z$, and $x\in[0,1]$, let
$c_{n,p,q}$, $\xi_{n,p,q}$ and $A_{ p,q}(x)$ be defined as in Proposition \ref{prop:zakchirp}.
Then, for every $f\in L^2(\R)$,
\begin{equation}
\label{eq:zakfrft0}
\ff_\alpha f(\xi)=c_\alpha e^{-i\pi\frac{p}{q}\xi^2}\int_0^1 e^{-2i\pi\frac{\xi}{\sin\alpha} x-i \pi \frac{p}{q} x^2} 
\sum_{ n \in A_{p,q}(x) } 
c_{n,p,q}\  Zf \left(x,\frac{\xi}{\sin\alpha}+\xi_{n,p,q}(x) \right)\,\de x.
\end{equation}
The identity has to be taken in the $L^2(\R)$ sense.
\end{theorem}

Note that $\sin^2\alpha=\cot^2\alpha\cos^2\alpha=(1-\sin^2\alpha)\frac{p^2}{q^2}$ thus $\sin^2\alpha=\frac{p^2}{p^2+q^2}$
and as $\sgn \sin\alpha=\sgn\cot\alpha=\sgn p$, $\sin\alpha=\frac{p}{\sqrt{p^2+q^2}}$
so that \eqref{eq:zakfrft0} can easily be written in terms of $p,q$ only.

\begin{proof}
Let us now fix $\omega\in\R$ and $f\in\ss(\R)$ and define $f_\omega$ as $f_\omega(t)=f(t)e^{-2i\pi\omega t}$.
Then $Zf_\omega(x,\xi)=Zf(x,\xi+\omega)e^{-2i\pi\omega x}$. Thus, applying \eqref{eq:zak} to $f_\omega$ leads to
\begin{equation}
\label{eq:zakfrft1}
\int_{\R} f(t) e^{-i\frac{p}{q}t^2}e^{-2i\pi\omega t}\,\de t 
=  \int_0^1 e^{-2i\pi\omega x-i \pi \frac{p}{q} x^2} \sum_{ n \in A_{p,q}(x) } 
c_{n,p,q}\  Zf \left(x,\omega+\xi_{n,p,q}(x) \right)\,\de x.
\end{equation}
But
$$
\ff_\alpha f(\xi)=c_\alpha e^{-i\pi\frac{p}{q}\xi^2}\int_{\R} f(t) e^{-i\frac{p}{q}t^2-2i\pi\xi t/\sin\alpha}\,\de t 
$$
so that \eqref{eq:zakfrft1} reads
\begin{equation}
\label{eq:zakfrft2}
\ff_\alpha f(\xi)=c_\alpha e^{-i\pi\frac{p}{q}\xi^2}\int_0^1 e^{-2i\pi\frac{\xi}{\sin\alpha} x-i \pi \frac{p}{q} x^2} \sum_{ n \in A_{p,q}(x) } 
c_{n,p,q}\  Zf \left(x,\frac{\xi}{\sin\alpha}+\xi_{n,p,q}(x) \right)\,\de x.
\end{equation}

It remains to extend this identity from $\ss(\R)$ to $L^2(\R)$. For sake of simplicity, we will assume that $p\geq 0$.

As $\ff_\alpha$ is continuous from $L^2(\R)\to L^2(\R)$, it is enough to check that the right hand side of
\eqref{eq:zakfrft2} is continuous as well.

To start, multiplication by bounded functions and dilations are continuous from $L^2(\R)\to L^2(\R)$ so that 
it is enough to check that the functional $F$ defined by
$$
F[f](\xi)=\int_0^1 e^{-2i\pi\xi x-i \pi \frac{p}{q} x^2} \sum_{ n \in A_{p,q}(x) } 
c_{n,p,q}\  Zf \left(x,\xi+\xi_{n,p,q}(x) \right)\,\de x
$$
is continuous from $L^2(\R)\to L^2(\R)$.

Let us first re-write this functional. 
To do so, note that $n\in A_{p,q}(x)$ if and only if 
$-pq/2-p\leq-p(x+q/2)\leq n\leq -p(x+q/2)+q\leq -pq/2+q$. 
Let $\ff_{p,q}=\Z\cap[pq/2-p,pq/2+q]$ then, for each $n\in \ff_{p,q}$ there is a subset $B_{p,q}(n)$ of $[0,1]$ such that
$n\in A_{p,q}(x)$ if and only if $x\in B_{p,q}(n)$.
Then $F[f]=\dst\sum_{n\in \ff_{p,q}} c_{n,p,q} F_n[f]$
where 
$$
F_n[f](\omega)=\int_{B_{p,q}(n)} e^{-2i\pi\omega x}e^{-i \pi \frac{p}{q} x^2}
Zf \left(x,\omega+\xi_{n,p,q}(x) \right)\,\de x.
$$
As $\ff_{p,q}$ is finite, it is enough to
show that each $F_n$ is continuous. Further, if $\omega=\eta+k$ with $k\in\Z$ and $\eta\in[0,1)$ then
\begin{eqnarray*}
F_n[f](\eta+k)&=&\int_{B_{p,q}(n)} e^{-2i\pi\eta x}e^{-i \pi \frac{p}{q} x^2}
Zf \left(x,\eta+\xi_{n,p,q}(x) \right)e^{-ikx}\,\de x\\
&=&\int_0^1 \mathcal{Z}(x,\eta)e^{-ikx}\,\de x
\end{eqnarray*}
with
$$
\mathcal{Z}(x,\eta)=\mathbf{1}_{B_{p,q}(n)}(x)e^{-i \pi \frac{p}{q} x^2}Zf \left(x,\eta+\xi_{n,p,q}(x) \right).
$$
Thus
\begin{eqnarray*}
\norm{F_n}^2_{L^2(\R)}&=&\int_0^1\sum_{k\in\Z}|F_n[f](\eta+k)|^2\,\de\eta\\
&=&\int_0^1\int_0^1|\mathcal{Z}(x,\eta)|^2\,\de x\,\de\eta
\end{eqnarray*}
with Plancherel.

It remains to notice that, for any fixed compact set $A\subset\R^2$, $f\to Z[f]$ is bounded $L^2(\R)\to L^2(A)$
thus $f\to Zf \left(x,\eta+\xi_{n,p,q}(x) \right)$ is bounded $L^2(\R)\to L^2(Q)$ and so is $f\to\mathcal{Z}$.
\end{proof}

\subsection{Application to the generalized Pauli problem}
\label{sec:pauli}
Fix $p\in\Z$, $q\in\N$ relatively prime and $\alpha$ given by $\cot\alpha=\frac{p}{q}$.

Let $\T=\R/\Z$. We will identify $\T$ with $[0,1)$. Consider the line $\Gamma_{p,q}(\xi)$
of slope $\dst\frac{p}{q}$ through the point $\left(0,\frac{p}{2}+\frac{\xi}{\sin\alpha}\right)$ in $\T^2$.
Note that this line has finite length $\sqrt{p^2+q^2}$ since its slope is rational. This well known fact can be proved as follows:
first the length does not depend on the starting point so we may assume that $\frac{p}{2}+\frac{\xi}{\sin\alpha}=0$.
Then the line is the periodization of the segment starting at $(0,0)$ of slope $p/q$ reaching the first point with integer coordinates,
that is $(p,q)$.

Let $\pi\,:\T^2\to\T$ be the projection on the second coordinate $\pi(x,v)=v$ where we identify $\T=[0,1)$.
For $(x,\omega)\in\Gamma_{p,q}(\xi)$, let $n(x,\omega)\in\Z$ be the number defined as follows:
$\omega=x+\frac{p}{2}+\frac{\xi}{\sin\alpha}+n(x,\omega)$. Then we may write \eqref{eq:zakfrft0} as
$$
\ff_\alpha f(\xi)=\frac{q c_\alpha}{\sqrt{p^2+q^2}} e^{-i\pi\left(\frac{p}{q}+\frac{q}{p\sin^2\alpha}\right)\xi^2}
\int_{\Gamma_{p,q}(\xi)} e^{-i\pi\frac{q}{p}\left(\pi(u)-\frac{p}{2}\right)^2} c_{n(u),p,q}\  Zf (u)\,\de u.
$$
In particular, this identity allows to relate the support of $\ff_\alpha f$ to the support of $Z[f]$:
\begin{equation}
\label{eq:suppzak}
\supp\ff_\alpha f\subset\{\xi\,:\Gamma_{p,q}(\xi)\cap\supp Zf\not=\emptyset\}.
\end{equation}

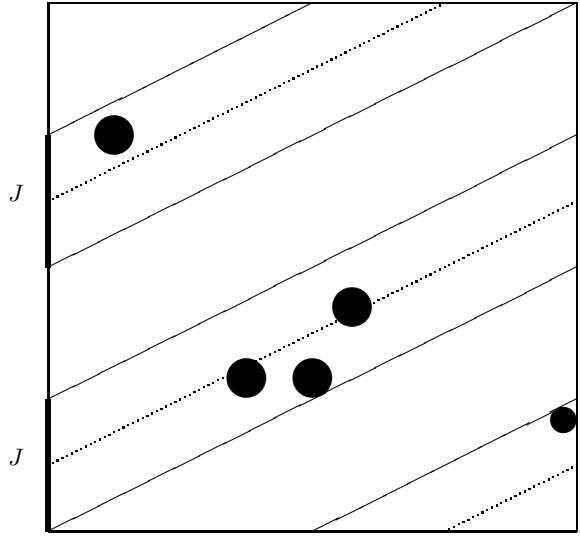
\begin{figure}
\begin{center}
\begin{picture}(250,250)
\drawline(25,25)(25,225)
\drawline(25,225)(225,225)
\drawline(225,25)(225,225)
\drawline(225,25)(25,25)
\drawline(25,25)(225,125)
\drawline(25,75)(225,175)
\drawline(25,125)(225,225)
\drawline(25,175)(125,225)
\drawline(125,25)(225,75)
\dottedline{2}(25,50)(225,150)
\dottedline{2}(25,150)(175,225)
\dottedline{2}(225,50)(175,25)

\put(125,83){\circle*{20}}
\put(100,83){\circle*{20}}
\put(140,110){\circle*{20}}
\put(50,175){\circle*{30}}
\put(220,67){\circle*{10}}

\put(10,50){$J$}
\put(10,150){$J$}

\linethickness{0.6mm}
\put(25,25){\line(0,1){50}}
\put(25,125){\line(0,1){50}}
\end{picture}
\caption{In this figure $p=2,q=1$. The support of $Z[f]$ is the union of the black discs. Then the set $J$
is such that, for each $\eta\in J$, the (periodic) line starting at $(0,\eta)$ with slope $q/p=1/2$ ($\cotan\alpha=p/q$)
intersects at least once the support of $Zf$ (one such line is drawn dotted). Here $J=[0,1/4]\cup[1/2,3/4]$
Finally, the support of $\ff_\alpha[f]$
is then included in the set of all $\xi$ such that $\xi/\sin\alpha+p/2\in J+\Z=\bigcup_{k\in\Z}[k/2,(k+1/2)/2]$.
As $p=2,q=1$, $\sin\alpha=2/\sqrt{5}$, thus $\supp\ff_\alpha[f]\subset\bigcup_{k\in\Z}[k/\sqrt{5},(k+1/2)/\sqrt{5}]$.}
\label{fig:diffplanobj}
\end{center}
\end{figure}

We can now extend the idea of Janssen \cite{janss2} to provide counterexamples for the generalized Pauli problem.

\begin{corollary}
\label{cor:supp}
Let $r_1,\ldots,r_m\in\Q$ be $m$ different rational numbers and $N\in\N$. For $j=1,\ldots,m$ let $\alpha_j\in[0,\pi)$ be defined
via $\cot\alpha_j=r_j$. Then there exists $f_1,\ldots,f_N\in L^2(\R)$ such that,
for $j=1,\ldots,m$ and $k\not=\ell=1,\ldots,N $,
\begin{equation}
\label{eq:corsupp}
\supp \ff_{\alpha_j}f_k\cap\supp\ff_{\alpha_j}f_\ell=\emptyset.
\end{equation}

In particular, for every $j=1,\ldots,m$ and for every $c_1,\ldots,c_N\in\C$ with $|c_k|=1$, 
\begin{equation}
\label{eq:pauli}
\abs{\ff_{\alpha_j}\ent{\sum_{k=1}^mc_kf_k}}=\abs{\ff_{\alpha_j}\ent{\sum_{k=1}^mf_k}}
\end{equation}
so that $\{Q_{\alpha_j},j=1,\ldots,m\}$ is not informationnaly complete with
respect to pure states.
\end{corollary}

\begin{proof} 
Of course \eqref{eq:pauli} follows directly from \eqref{eq:corsupp} since then
$$
\ff_{\alpha_j}\ent{\sum_{k=1}^mc_kf_k}(\xi)=\begin{cases}c_k\ff_{\alpha_j}[f_k](\xi)&\mbox{if }\xi\in\supp\ff_{\alpha_j}[f_k]\\
0&\mbox{otherwise}
\end{cases}.
$$

Let us now prove \eqref{eq:corsupp}.
Write each $r_j=p_j/q_j$ as an irreducible fraction with $q_j>0$. Let $L_j=\sqrt{p_j^2+q_j^2}$
and $L=\sum L_j$.

The proof is based on the following observation:
Let $I\subset[0,1)$ be an open interval, $r=p/q\in\Q$
and $\Gamma_{p_j,q_j}(I)=\bigcup_{\xi\in I}\Gamma_{p_j,q_j}(\xi)$. Note that $\Gamma_{p_j,q_j}(I)$ has area $L_j|I|$.
 Note that,
if $J\subset[0,1)$ ia an interval such that $(\{0\}\times J)\cap\Gamma_{p_j,q_j}(I)=\emptyset$ then 
$\Gamma_{p_j,q_j}(I)\cap\Gamma_{p_j,q_j}(J)=\emptyset$. Moreover, the existence of such a $J$ is guarantied when $\Gamma_{p_j,q_j}(I)\not=\T^2$.

Further $\bigcup_j \Gamma_{p_j,q_j}(I)$ has area $\leq\sum_j L_j |I|=L|I|$. If $L|I|<1$ then $\bigcup_j \Gamma_{p_j,q_j}(I)\not=\T^2$.
Therefore there exists $J\subset(0,1)$ such that $J\cap \bigcup_j \Gamma_{p_j,q_j}(I)=\emptyset$ and therefore
$\Gamma_{p_j,q_j}(I)\cap\Gamma_{p_j,q_j}(J)=\emptyset$. Write $I_1=I$ and $I_2=J$.
If further $L(|I_1|+|I_2|)<1$ then $\bigcup_j \Gamma_{p_j,q_j}(I_1)
\cup \bigcup_j \Gamma_{p_j,q_j}(I_2)\not=\T^2$. one can thus find $I_3\subset[0,1)$ such that 
$\Gamma_{p_j,q_j}(I_k)\cap\Gamma_{p_j,q_j}(I_\ell)=\emptyset$ if $k\not=\ell$.

In general, we are thus able to find open intervals $I_1,\ldots,I_m\subset(0,1)$ such that 
\begin{equation}
\label{eq:gammas}
\Gamma_{p_j,q_j}(I_k)\cap\Gamma_{p_j,q_j}(I_\ell)=\emptyset\mbox{ if }k\not=\ell.
\end{equation}
Next, note that $K_k:=\bigcap\Gamma_{p_j,q_j}(I_k)$ has non-empty interior as it contains a neighborhood of $\{0\}\times I_k$. Let
$Z_k$ be any function in $L^2(Q)$ and let $f_k\in L^2(\R)$ be given by $Z[f_k]=Z_k$. Then \eqref{eq:suppzak}
and \eqref{eq:gammas} imply \eqref{eq:corsupp}.
\end{proof}

\section{Approximate Pauli Problem}

\begin{theorem}
Leq $0\leq \alpha_1< \cdots<\alpha_n \leq \frac{\pi}{2} $ and $T>0$.
Let $f_1, ... f_n \in L^\infty(\R)$ with $\supp f_j\subset\ent{-T,T}$ and $f_j\geq0$. Then, for every $\eps>0$, there exists a function $\ffi \in L^2(\R)$ 
such that, for every $k=1,\ldots, n$,
$$
\norm{\abs{\ff_{\alpha_k}[\ffi]}-f_k}_{L^\infty(\ent{-T,T})} \leq \eps.
$$ 
\end{theorem}

\begin{proof}
It is enough to show that, given $\eps > 0$, for every $j\not=k \in \{1 \cdots n\}$, there exists $\ffi_k \in L^2(\R)$ such that
$\ff_{\alpha_k}[\ffi_k]=f_k$ on $\ent{-T,T}$ and $\norm{\ff_{\alpha_j}[\ffi_k]}_{L^\infty(\ent{-T,T})} \ieg \varepsilon/(n-1)$.
Taking $\ffi=\sum \ffi_k$ then gives the result. 

Without loss of generality, it is enough to construct $\ffi_1$.

Let $\omega\in\R$ be a parameter that we will fix later.
As $f_1\in L^1(\R)$, we may define $h_\omega=\ff_{-\alpha_1} \ent{f_1(t) e^{2i\pi \omega t}}$.
Note that $\abs{\ff_{\alpha_1}[h_\omega]}=f_1$ on $[-T,T]$. Moreover, for $k=2,\ldots,n$, 
$\ff_{\alpha_k} [h_\omega]=\ff_{\alpha_k-\alpha_1} \ent{f_1(t) e^{ 2 i\pi \omega t}}$. 

Let $\tilde\alpha_k=\alpha_k-\alpha_1$. Then, for $\xi \in \R$,
\begin{eqnarray*}
\abs{\ff_{\tilde\alpha_k}\ent{f_1(t) e^{i \omega t}}(\xi) } 
&=& \abs{ c_\alpha \int_\R f_1 (x) e^{-i\pi \cot{\alpha} x^2} e^{2 i \pi x \left(\omega-\frac{\xi}{\sin{\alpha}} \right)}\,\de x }
=\abs{ \int_\R u_k(x) e^{2i \pi x \left( \omega-\frac{\xi}{\sin{\alpha}} \right) } dx}
= \abs{\widehat{u_k} \left( \frac{\xi}{\sin{\tilde\alpha_k}}-\omega \right)} 
\end{eqnarray*}
where $u_k(x)=f_1 (x) e^{-i\pi \cot{\alpha} x^2}$ is an $L^1(\R)$ function.
According to the Riemann-Lebesgue Lemma, $\widehat{u_k}(\eta)\to 0$ as $\eta\to\pm\infty$.
Thus, there exists $A$ such that, if $\eta>A$, $\dst|\widehat{u_k}(\eta)|<\frac{\eps}{n-1}$.
But then, if $\omega\geq A+T/\sin\tilde\alpha_k$ and $|\xi|<T$, $\omega-\frac{\xi}{\sin\tilde\alpha_k}>A$
thus $|\ff_{\alpha_k}[h_\omega](\xi)|<\dst\frac{\eps}{n-1}$.
It remains to chose $\omega=A+T/\sin\tilde\alpha_2>A+T/\sin\tilde\alpha_k$ and $\ffi_1=h_\omega$.
\end{proof}

\begin{remark}
The proof actually shows a bit more, namely that there are infinitely many solutions to the problem
that are not constant multiples of one an other. Indeed, one could as well take any $\ffi=\sum c_k\ffi_k$
with $|c_k|=1$.

However, an inspection of the proof of the theorem shows that the function we build has a huge support
(if $\alpha_1=0$).
Thus, if one imposes additional support constraint (as would be imposed in real life applications),
the theorem may no longer be valid.
\end{remark}

\appendix

\section{The modulus of the coefficients $c_{n,p,q}$}

The factor $c_{n,p,q}$ appearing in Proposition \ref{prop:zakchirp}
is similar to a Gauss sum. As for Gauss sums, it is possible to compute their modulus.

\begin{lemma} Let $p\in\Z$, $q\in\N$ be relatively prime integers.
Fon $n\in\Z$, let
\[ 
c_{n,p,q}=\frac{1}{q} \sum_{k=0}^{q-1} (-1)^{kp} e^{i \pi \frac{p}{q} k^2}e^{-2i \pi \frac{nk}{q} }.
\]
Then $\dst\abs{c_{n,p,q}}= \frac{1}{\sqrt{q}}$.
\end{lemma}

\begin{proof}
Indeed, setting $\omega=\exp(2i \pi /q)$, we get
\begin{eqnarray*}
q^2 \abs{c_{n,p,q}}^2 
&=& \sum_{0\ieg l,k<q} (-1)^{p(k-l)} w^{n(l-k)+\frac{p}{2}(k^2-l^2)}
=\sum_{-q < s < q} (-1)^{ps} \omega^{ns} \sum_{0 \leq k,l<q\,: k-l=s} \omega^{\frac{p}{2}(k^2-(k-s)^2)}\\
&=&\sum_{-q < s < q} (-1)^{ps} \omega^{ns-\frac{p}{2}s^2} \sum_{k-l=s, 0 \ieg k,l<q} \omega^{skp} .
\end{eqnarray*}
Now note that, for $s\geq 0$ the second sum is over $k=s\,\ldots,q-1$ while for $s<0$, this sum is over
$k=0,\ldots,q-1+s$. Finally, for $s=0$ this sum is just $=q$. The sum thus splits into three parts

\begin{eqnarray*}
q^2 \abs{c_{n,p,q}}^2 
&=&\sum_{s=-q}^{s=-1} (-1)^{ps} \omega^{ns-\frac{p}{2}s^2} \sum_{k=0}^{q-1+s} \omega^{skp}
+q+\sum_{s=1}^q (-1)^{ps} \omega^{ns-\frac{p}{2}s^2} \sum_{k=s}^{q-1} \omega^{skp}\\
&=&\sum_{s=1}^q (-1)^{ps-pq} \omega^{ns-nq-\frac{p}{2}s^2+pqs-\frac{p}{2}q^2} \sum_{k=0}^{s-1} \omega^{(s-q)kp}
+q+\sum_{s=1}^q (-1)^{ps} \omega^{ns-\frac{p}{2}s^2} \sum_{k=s}^{q-1} \omega^{skp}
 \end{eqnarray*}
changing $s\to s-q$ in the first sum.
But $\omega^q=1$ and $\omega^{\frac{p}{2}q^2}=(-1)^{pq}$ thus 
\begin{eqnarray*}
q^2 \abs{c_{n,p,q}}^2 
&=&\sum_{s=1}^q (-1)^{ps} \omega^{ns-\frac{p}{2}s^2-\frac{p}{2}q^2} \sum_{k=0}^{s-1} \omega^{skp}
+q+\sum_{s=1}^q (-1)^{ps} \omega^{ns-\frac{p}{2}s^2} \sum_{k=s}^{q-1} \omega^{skp}\\
&=&q+\sum_{s=1}^q (-1)^{ps} \omega^{ns-\frac{p}{2}s^2} \sum_{k=0}^{q-1} \omega^{skp}.
 \end{eqnarray*}
Finally, it remains to notice that $\dst\sum_{k=0}^{q-1}\omega^{skp}=0$ unless $q$ divides $sp$. As
$q$ and $p$ are relatively prime and  $s \in {1, \cdots, q-1}$, this can not happen, thus
$q^2 \abs{c_{n,p,q}}^2=q$.
\end{proof}

\section*{Acknowledgements}
The first author kindly acknowledge financial support from the French ANR programs ANR
2011 BS01 007 01 (GeMeCod), ANR-12-BS01-0001 (Aventures).
This study has been carried out with financial support from the French State, managed
by the French National Research Agency (ANR) in the frame of the ”Investments for
the future” Programme IdEx Bordeaux - CPU (ANR-10-IDEX-03-02).

\end{document}